\documentclass[12pt,fleqn]{amsart}

\usepackage{amssymb}

\newtheorem{theorem}{Theorem}[section]
\newtheorem{lemma}[theorem]{Lemma}
\newtheorem{proposition}[theorem]{Proposition}
\newtheorem{problem}[theorem]{Problem}
\newtheorem{corollary}[theorem]{Corollary}
\theoremstyle{definition}

\newtheorem{example}[theorem]{Example}

\theoremstyle{remark}

\numberwithin{equation}{section}

\begin{document}

\title[Embeddings of  $C(K)$ spaces]{On positive embeddings of $C(K)$ spaces}

\author[G.\ Plebanek]{Grzegorz Plebanek}
\address{Instytut Matematyczny, Uniwersytet Wroc\l awski}
\email{grzes@math.uni.wroc.pl}
\thanks{Partially suported by 
the EPSRC grant No EP/G068720/1 (2009-2010) and by MNiSW Grant N N201 418939 (2010--2013).
I wish to thank Mirna D\v{z}amonja for her hospitality during my visits to the University of East Anglia and for several stimulating discussions related to the subject of this paper. I am also grateful to Witold Marciszewski for valuable comments.}




\begin{abstract}
We investigate isomorphic embeddings $T: C(K)\to C(L)$ between Banach spaces of continuous functions.
We show that if such an embedding $T$ is a positive operator then 
$K$ is an image of $L$ under a upper semicontinuous set-function having finite values. Moreover we show
that $K$ has a $\pi$-base of sets which closures are continuous images of compact subspaces of $L$.
Our results imply in particular that if $C(K)$ can be positively embedded into $C(L)$ then 
 some topological properties of $L$, such as countable tightness of Frechetness, are inherited by the space $K$.

We show that some arbitrary isomorphic embeddings $C(K)\to C(L)$ can be, in a sense,
reduced to positive embeddings.
\end{abstract}

\maketitle

\newcommand{\con}{\mathfrak c}
\newcommand{\eps}{\varepsilon}
\newcommand{\alg}{\mathfrak A}
\newcommand{\algb}{\mathfrak B}
\newcommand{\algc}{\mathfrak C}
\newcommand{\ma}{\mathfrak M}
\newcommand{\pa}{\mathfrak P}
\newcommand{\BB}{\protect{\mathcal B}}
\newcommand{\AAA}{\mathcal A}
\newcommand{\CC}{{\mathcal C}}
\newcommand{\FF}{{\mathcal F}}
\newcommand{\GG}{{\mathcal G}}
\newcommand{\LL}{{\mathcal L}}
\newcommand{\UU}{{\mathcal U}}
\newcommand{\VV}{{\mathcal V}}
\newcommand{\HH}{{\mathcal H}}
\newcommand{\DD}{{\mathcal D}}
\newcommand{\RR}{\protect{\mathcal R}}
\newcommand{\ide}{\mathcal N}
\newcommand{\btu}{\bigtriangleup}
\newcommand{\hra}{\hookrightarrow}
\newcommand{\ve}{\vee}
\newcommand{\we}{\cdot}
\newcommand{\de}{\protect{\rm{\; d}}}
\newcommand{\er}{\mathbb R}
\newcommand{\qu}{\mathbb Q}
\newcommand{\supp}{{\rm supp} }
\newcommand{\card}{{\rm card} }
\newcommand{\wn}{{\rm int} }
\newcommand{\ult}{{\rm ULT}}
\newcommand{\vf}{\varphi}
\newcommand{\osc}{{\rm osc}}
\newcommand{\ol}{\overline}
\newcommand{\me}{\protect{\bf v}}
\newcommand{\ex}{\protect{\bf x}}
\newcommand{\stevo}{Todor\v{c}evi\'c} 
\newcommand{\cc}{\protect{\mathfrak C}}
\newcommand{\scc}{\protect{\mathfrak C^*}}
\newcommand{\lra}{\longrightarrow}
\newcommand{\sm}{\setminus}
\newcommand{\uhr}{\upharpoonright}

\newcommand{\sub}{\subseteq}
\newcommand{\ms}{$(M^*)$} 
\newcommand{\m}{$(M)$} 
\newcommand{\MA}{MA$(\omega_1)$} 
\newcommand{\clop}{\protect{\rm Clop} }
\section{Introduction}
For a compact space $K$ we denote by $C(K)$ the Banach space of real--valued continuous functions
with the usual supremum norm. In the sequel, $K$ and $L$ always denote compact Hausdorff spaces.

Let $T:C(K)\to C(L)$ be an isomorphism of Banach spaces.
By the classical  Kaplansky theorem, if $T$ is an order-isomorphism, i.e.\  $g\ge 0$ if and only if $Tg\ge 0$
for every $g\in C(K)$, then $K$ and $L$ are homeomorphic; see 7.8 in Semadeni's book \cite{Se71} for further references.
On the other hand, if $C(K)$ and $C(L)$ are isomorphic as Banach spaces then $K$ may be topologically different from $L$. 
 For example, by Miljutin's theorem, $C[0,1]$ is isomorphic to $C(2^\omega$) as well as to
any $C(K)$, where $K$ is uncountable metric space, see \cite{Se71}, 21.5.10 or \cite{Pe68}.

In a present paper we consider isomorphic embeddings $T:C(K)\to C(L)$ which are not necessarily onto but are positive operators,
i.e.\ 
\[\mbox{if } g\in C(K) \mbox{ and } g\ge 0 \mbox{ then } Tg\ge 0.\]
Elementary examples show that even if such an operator $T$ is onto then $K$ may  not be homeomorphic to $L$, see e.g.\ Example \ref{rq:e} below. 

Our main objective is to determine how $K$ is related to $L$ whenever $C(K)$ admits a positive embedding into $C(L)$.
We show that in such a case for some natural number $p$ there is a function $\vf:L\to [K]^{\le p}$ which is upper semicontinuous
and onto (that is $K$ is the union of values of $\vf$). This implies that some topological properties of $L$, such as countable tightness
or Frechetness, are inherited by $K$. We moreover prove that $K$ has a $\pi$-base of sets with closures being continuous images
of subspaces of $L$. Our results offer partial generalizations of a theorem due to Jarosz \cite{Ja84}, who proved that
if $T:C(K)\to C(L)$ is an isomorphic embedding which is not necessarily positive but
satisfies $||T||\cdot ||T^{-1}||<2$ then $K$ is a continuous image of a compact subspace of $L$.

Let us recall the following open problem  related to isomorphic embeddings of $C(K)$ spaces and the class of Corson compacta ({see the next section for the terminology}):

\begin{problem}\label{in:2}
Suppose that $T:C(K)\to C(L)$ is an isomorphic embedding and $L$ is Corson compact.
Is $L$ necessarily Corson compact? 
\end{problem}

The answer to \ref{in:2} is positive under
Martin's axiom and the negation of continuum hypothesis, see Argyros et.\ al.\ \cite{AMN}. 
However, the problem remains open in  ZFC even if the operator $T$ in question is onto, see
6.45 in Negrepontis \cite{Ne84} or Question 1 in Koszmider \cite{Ko07}.
In \cite{MP10} we proved  that the answer to \ref{in:2} is positive under a certain additional measure-theoretic
assumption imposed on the space $K$.

We have not been able to fully resolve Problem \ref{in:2} even for positive embeddings but we show in section 4
that the answer is affirmative whenever  the space $K$ under consideration is homogeneous.

Our approach to analysing embeddings $T:C(K)\to C(L)$ follows Cambern \cite{Ca67} and Pe\l czy\'nski \cite{Pe68}:
we consider the conjugate operator $T^*:C(L)^*\to C(K)^*$ and a mapping
$L\ni y\to T^*\delta_y$ which sends points of $L$ to measures on $K$ and is $weak^*$ continuous. 
The main advantage of dealing with positive $T$ here is that the measures of the form $T^*\delta_y$ are nonnegative. In fact, as
it is explained in the final section, what is really crucial here is the continuity of the mapping $L\ni y\to ||T^*\delta_y||$.
In a recent preprint \cite{Pl13}  we were able to extend some of the results presented here to the case of arbitrary isomorphisms
between spaces of continuous functions.

The paper is organized as follows. In the next section we collect some standard fact on operators and
the $weak^*$ topology of $C(K)^*$, we also recall some concepts from general topology.
In section 3, we consider several properties of compact spaces that are preserved by taking the images under
upper semicontinuous finite-valued maps.

In section 4 we show that a positive embedding $T:C(K)\to C(L)$ gives rise 
to a natural map $L\to [K]^{<\omega}$ and conclude from this our main results.
Section 5 contains a few comments related to the results;
in section 5 we consider arbitrary embeddings $C(K)\to C(L)$ for which the above mentioned function
$y\to ||T^*\delta_y||$ is continuous.

\section{Preliminaries}
Throughout this paper  we tacitly assume that $K$ and $L$  denote compact Hausdorff spaces. 
The dual space $C(K)^*$ of the Banach space $C(K)$ is identified with $M(K)$ --- the
space of all signed Radon measures of finite variation; we use the symbol $M_1(K)$ to denote the unit ball of $M(K)$
and $P(K)$ for the space of Radon probability measures on $K$; every $\mu\in P(K)$ is an inner regular probability measure defined on the Borel $\sigma$--algebra $Bor(K)$ of $K$.
The spaces $M_1(K)$ and $P(K)$ are always equipped with the $weak^*$
topology inherited from $C(K)^*$, i.e.\ the topology making all the functionals
$\mu\to \int g\;{\rm d}\mu$,
continuous for $g\in C(K)$. We usually write $\mu(g)$ rather than $\int_K g\de\mu$. 

We shall frequently use the following simple remark: for every closed set  $F\sub K$, the set 
\[\{\mu\in P(K): \mu(F)<r\},\]
is $weak^*$ open in $P(K)$. 

For
any $x\in K$ we write $\delta_x\in P(K)$ for the corresponding Dirac measure; recall that
$\Delta(K)=\{\delta_x: x\in K\}$ is a subspace of $P(K)$ which is homeomorphic to $K$.

A linear operator $T:C(K)\to C(L)$ is an isomorphic  embedding if the are positive constants 
$m_1,m_2$ such that
\[ m_1||g||\le ||Tg||\le m_2 ||g||,\]
for every $g\in C(K)$ (so that $||T||\le m_2$ and $||T^{-1}||\le 1/m_1$). 
When we say `embedding' we always mean an isomorphic embedding which is not necessarily surjective.
 
To every bounded operator $T:C(K)\to C(L)$ we can associate a conjugate operator $T^*$, where
\[T^*: M(L)\to M(K),\quad  T^*(\nu)(f)=\nu(Tf);\]
$T^*$ is surjective whenever $T$ is an isomorphic embedding.

A set $M\sub M(K)$ is said to be {\em $m$-norming}, where $m>0$, if
for every $g\in C(K)$ there is $\mu\in M$ such that $|\mu(g)|\ge m\cdot ||g||$.

\begin{lemma}\label{pr:1}
If $T:C(K)\to C(L)$ is an embedding then $L\ni y\to T^*\delta_y\in M(K)$ is a continuous mapping; 
the set
\[T^*[\Delta_L]=\{T^*\delta_y: y\in L\},\]
is a weak$^*$ compact and $m$-norming  subset of $M(K)$, where $m=1/||T^{-1}||$.
\end{lemma}

\begin{proof}

If $(y_t)_t$ is a net in $L$ converging to $y$ then the measures 
$\delta_{y_t}$ 
converge to $\delta_y$ in the $weak^*$ topology
of $M(L)$, and $T^*\delta_{y_t}\to T^*\delta_y$ since the conjugate operator is $weak^*$-$weak^*$ continuous.

If $g\in C(K)$ and $||g||=1$, then $m\le ||Tg||$, so there is $y\in L$
such that $|Tg(y)|\ge m$, so $|Tg(y)|=|T^*\delta_y(g)|$; this shows that $ T^*[\Delta_L]$ is
$m$-norming.
\end{proof}

Every signed measure $\mu\in M(K)$ can be written as $\mu=\mu^+-\mu^-$, where
$\mu^+,\mu^-$ are nonnegative mutually singular measures. We write $|\mu|=\mu^++\mu^-$
for the total variation $|\mu|$ of $\mu$; the natural norm of $\mu\in M(K)=C(K)^*$ is defined as
 $||\mu||=|\mu|(K)$. The mapping $\mu\mapsto |\mu|$ is not $weak^*$ continuous, but the following
 holds.

\begin{lemma}\label{pr:3}
Let $(\mu_i)_{i\in I}$ be  a net in $M_1(K)$ converging to $\mu\in M_1(K)$.
Then 
\[|\mu|(g)\le\liminf_i |\mu_i|(g),\] 
for $g\in C(K)$, $g\ge 0$ (that is
the maping $\mu\mapsto |\mu|$ is lower semicontinuous). Moreover
\[|\mu|(g)\ge |\mu|(K)+\limsup_i|\mu_i|(g)-1,\]
whenever $g\in C(K)$, $0\le g\le 1$.

In particular,   the mapping 
\[S=\{\nu\in M(K):||\nu||=1\}\ni \nu\mapsto |\nu|\in P(K)\]
is weak$^*$ continuous.
\end{lemma}

\begin{proof}
For any fixed  $\eps>0$ there are disjoint closed sets $F,H$, such that
\[\mu^+(F)+\mu^-(H)>|\mu|(K)-\eps.\]
 Take a continuous function $h:K\to [-1,1]$, such that
$h=1$ on $F$ and $h=-1$ on $H$. Now for any continuous function $g:K\to [0,1]$ we have
\[|\mu|(g)-2\eps< \mu(hg)<\mu_i(hg)+\eps<|\mu_i|(g)+3\eps,\] 
for $i\ge i_0\in I$. This shows that $|\mu|(g)\le\liminf_i |\mu_t|(g)$.

The second statement follows from the first one applied to  $1-g$.

If $|\mu|(K)=1$ we get $|\mu|(g)\ge\limsup_i |\mu_i|(g)$, so $|\mu_i|(g)\to |\mu|(g)$, and thus
the final statement follows.
\end{proof}

Let us recall that a compact space $K$ is {\em Corson compact} if, for some cardinal number $\kappa$, 
$K$ is homeomorphic to a subset of the $\Sigma$--product of real lines
\[\Sigma(\er^\kappa)=\{x\in \er^\kappa: |\{\alpha: x_\alpha\neq 0\}|\le\omega\}.\]
Concerning Corson compacta and their role in functional analysis we refer the reader to a paper \cite{AMN}
by Argyros, Mercourakis and Negrepontis,  and to extensive surveys Negrepontis \cite{Ne84} and
 Kalenda \cite{Ka00}. 
 \medskip
 
 We shall now recall several concepts of countability-like condition for arbitrary topological spaces;
 if $K$ is any topological space then $K$ is said to be 
 
 \begin{itemize}
\item[(i)] {\em Frechet} if for every $A\sub K$ and
 $x\in\ol{A}$ there is a sequence $(a_n)_n$ in $A$ converging to $x$;
\item[(ii)] {\em sequential}
 if for every nonclosed set $A\sub K$ there is a sequence $(a_n)_n$ in $A$ converging to
 a point $x\in K\sm A$;
  \item[(iii)] {\em sequentially compact} if every sequence in $K$ has a converging subsequence.
  \end{itemize}
 
The tightness of a topological space $K$, denoted here by $\tau(X)$ is the least cardinal number
such that if every $A\sub K$ and $x\in\ol{A}$ there is a set $I\sub A$ with $|I|\le \tau(K)$
and such that $x\in\ol{I}$.

Every Corson compactum is a Frechet space and every Frechet space is clearly
sequential; the reader may consult  \cite{En} (e.g.\ p.\ 78 and 3.12.7 - 3.12.11) for further information.

\section{Finite valued maps}

In the sequel we consider, for a given pair of compact spaces $K$, $L$ and a fixed natural number
$p$, set-valued mappings $\vf:L\to [K]^{\le p}$. Such a mapping $\vf$ is said to be 
\begin{itemize}
\item[(a)]
{\em onto} if $\bigcup_{y\in L}\vf(y)=K$;
\item[(b)] {\em upper semincontinuous} if the set  
$\{y\in L: \vf(y)\sub U\}$
is open for every open $U\sub K$.
\end{itemize}

Clearly upper semicontinuity is equivalent to saying that the set
\[\vf^{-1}[F]:=\{y\in L: \vf(y)\cap F\neq\emptyset\},\]
is closed whenever $F\sub K$ is closed.
Semicontinuous mappings with finite values have been considered  by Okunev \cite{Ok02}
in connection with some problems in $C_p(X)$ theory.

The first assertion of the following auxiliary result is a particular case of Proposition 1.2 from
\cite{Ok02}; we enclose here a different self-contained argument.

\begin{lemma}\label{fvm:1}
Let $K$ and $L$ be compact spaces and suppose that for some natural number $p$ there is 
an upper semicontinuous $\vf:L\to [K]^{\le p}$ onto mapping. Then

\begin{itemize}
\item[(i)] $\tau(K)\le \tau(L)$;
\item[(ii)] if $L$ is a Frechet space then $K$ is Frechet;
\item[(iii)] if $L$ is a sequential space then $K$ is sequential too.
\item[(iv)] if $L$ is a sequentially compact then $K$ is sequentially compact too.
\end{itemize}

\end{lemma}

\begin{proof}
To prove part (i) fix $A\sub K$ and $x\in\ol{A}.$ For every $a\in A$ there is $y_a\in L$ such that $a\in\vf(y_a)$.
Denote by $\VV(x)$ some local neighbourhood base at $x\in K$; the set 
\[F=\bigcap_{V\in \VV(x)}\ol{\{y_a: a\in V\cap A\}}.\]
is nonempty and for any $y\in F$ we have $x\in\vf(y)$. Indeed, if $x\notin\vf(y)$ and $y\in F$ then
$\vf(y)\sub K\sm \ol{V}$ for some $V\in\VV(x)$. But then 
\[U=\{z\in L: \vf(z)\sub K\sm \ol{V}\},\] 
is a neighbourhood of $y$ which is disjoint from $\{y_a:a\in V\cap A\}$, which is a contradiction.

Now fix $y\in F$ and choose $V\in\VV(x)$ such that $\vf(y)\cap \ol{V}=\{x\}$. Then
\[y\in\ol{\{y_a:a\in A\cap V\}},\]
 so by the definition of tightness there is $I\sub A\cap V$ with $|I|\le\tau(L)$ and such that
$y\in \ol{\{y_a:a\in I\}}$. Now it suffices to check that $x\in\ol{I}$.

If we suppose that $x\notin \ol{I}$ then $\vf(y)\sub K\sm \ol{I}$ (by our choice of $V$) so
\[W=\{z\in L: \vf(z)\sub K\sm \ol{I}\},\]
 is an open neighbourhood of $z$ which is in contradiction with $W\cap\{y_a: a\in I\}=\emptyset$.

Suppose now that $L$ is a Frechet space. Then we can argue as above but this time we can choose
a sequence $a(n)\in A\cap V$ such that $y_{a(n)}\to y$. Then we can check that simply $a(n)$ form a
sequence that converges
to $x$. Otherwise there is a cluster point $x'\neq x$ of the sequence $(a(n))_n$ and we can find an open set
$V_1\sub V$ such that $x'\in V_1$ while $x\notin\ol{V_1}$. Again, 
\[W_1=\{z\in L: \vf(z)\sub K\sm \ol{V_1}\},\]
 is an open set containing $y$ and
$W_1\cap\{y_{a(n)}: a(n)\in V_1\}=\emptyset$, a contradiction since a subsequence $y_{a(n)}$
indexed by $a(n)\in V_1$ should converge to $y$. 

Part (iii) can be checked in a similar way, cf.\ \cite{Ok02}, Proposition 1.6.

Part (iv) follows by the following: for any sequence $x_n\in K$ we may choose $y_n\in L$ such that
$x_n\in\vf(y_n)$. Since $L$ is sequentially compact we can assume that $y_n$ converge to some $y\in L$.
Then every cluster point of the sequence $(x_n)_n$ must lie in $\vf(y)$ so $(x_n)_n$ has only finite number of
cluster points and therefore must contain a converging subsequence.
\end{proof}

\begin{lemma}\label{fvm:2}
Let $K$ and $L$ be compact spaces and let
$\vf:L\to [K]^{\le p}$ be an upper semicontinuous  mapping. If $F\sub K$ is closed
then the mapping $\psi:L\to [F]^{\le p}$ given by $\psi(y)=\vf(y)\cap F$ is semicontinuous too.
\end{lemma}

\begin{proof}
If $V\sub F$ is open in $F$ then $V=F\cap U$ for some open $U\sub K$; note that 
$\psi(y)\sub V$ is equivalent to $\vf(y)\sub U\cup (K\sm F)$. 
\end{proof}

\section{Positive embeddings}

In this section we investigate what can be said on a pair of compact space $K$ and $L$ assuming there
is a positive embedding $T: C(K)\to C(L)$. 
The positiveness of $T$ implies that
that $T^*\delta_y$ is a nonnegative measure on $K$ for every $y\in L$, which is crucial for the proofs given below.
The first lemma will show that, after a suitable reduction, we can in fact assume that every 
$T^*\delta_y$ is a probability measure on $K$.

\begin{lemma}\label{pe:1}
Suppose that $C(K)$ can be embedded into $C(L)$ by
a positive operator $T$ of norm one. Then there is a compact subspace
$L_0\sub L$ and a positive embedding $S: C(K)\to C(L_0)$ such that $S 1_K= 1_{L_0}$ and
$||S^{-1}||\le ||T^{-1}||$.
\end{lemma}

\begin{proof}
We can without loss of generality
assume that for some $m>0$ the operator $T$ satisfies  $m\le ||Tg||\le 1$ whenever $||g||=1$.

By a standard application of Zorn's lemma there is a minimal element $L_0$ in the family
of those compact subsets $L'\sub L$, which have the property that
\[||Tg||=||(Tg)\uhr L'|| ,\]
for every $g\in C(K)$. 
Letting $h=(T 1_K){\uhr L_0}$ we have  $h\ge m$; indeed, otherwise
$U=\{y\in L_0:h(y)<m\}$ is a nonempty open subset of $L_0$, so $L_1=L_0\sm U$ is a proper closed
subspace of $L_0$. On the other hand, if $g\in C(K)$ and $||g||=1$, then 
$Tg(y_1)\ge m$ for some $y_1\in L_0$; since $Tg\le T 1_K=h$ we get $y_1\in L_1$. This
shows that $L_0$ is not minimal, a contradiction.

We can now define $S: C(K)\to C(L_0)$ letting $Sg$ be the function $(1/h)(Tg)$ restricted to
${ L_0}$. Then $S1_K=1_{L_0}$; as $S$ is positive this implies $||S||=1$.  Moreover, $||Sg||\ge ||Tg||$ by the choice of $L_0$.
\end{proof}

\begin{proposition}\label{pe:2}
Suppose that $T: C(K)\to C(L)$ is a positive embedding such that $T1_K=1_L$ and $||T^{-1}||=1/m$.
 For an $r\in (0,1]$ and $y\in L$ we set
\[ \vf_r(y)=\{x\in K: \nu_y(\{x\})\ge r\},\]
where $\nu_y=T^*\delta_y$.  Then for every $r\in (0,1]$

\begin{itemize}
\item[(i)] $\vf_r$  takes its values in $[K]^{\le p}$, where $p$ is the integer part of $1/r$, and is upper semicontinuous; 
\item[(ii)] the set $\bigcup_{y\in L}\vf_r(y)$ is closed in $K$; 
\item[(iii)] the mapping $\vf_m$ is onto $K$.
\end{itemize}
\end{proposition}

\begin{proof}
Note first that since $\nu_y(g)=Tg(y)\ge 0$ for every
$g\ge 0$, we have $\nu_y\in M^+(K)$. Moreover, $\nu_y(K)=\nu_y(1_K)=T 1_K(y)=1$, so every
$\nu_y$ is in fact a probability measure on $K$.
\medskip

It is clear that $\vf_r$ has at most $1/r$ elements so to verify (i) we shall check upper semicontinuity.

Take an open set  $U\sub K$ and $y\in L$ such that $\vf(y)\sub U$.
Writing $F= K\sm U$, for every 
$x\in F$ we have $\nu_y(\{x\})<m$  so there is
an open set $U_x\ni x$ such that $\nu_y(\ol{U_x})<m$. This defines an open cover $U_x$
of a closed set $F$ so we have $F\sub \bigcup_{x\in I}U_{x}$ for some finite set $I\sub F$.
Now 
\[V=\{z\in L: \nu_z(\ol{U_x})<m\mbox{ for every } x\in I\},\]
is an open  neighbourhood of $y$ and $\vf(z)\sub U$ for all $z\in V$.
\medskip

To check (ii) let us fix $r>0$ and write $K(r)=\bigcup_{y\in L} \vf_r(y)$. Take any $x\in K\sm K(r)$;
for every $y\in L$ we have $\nu_r(\{x\})<r$ so there is an open set $U_y$ in $K$ such that
$x\in U_y$ and $\nu_y(\ol{U_y})<r$. In turn, there is an open set $V_y$ in $L$ such that
for every $y'\in V_y$ we have $\nu_{y'}(\ol{U_y})<r$. We can choose a finite set $J\sub L$ so that 
$\{V_y: y\in J\}$ is a cover of $L$. Letting
\[ U=\bigcap_{y\in J} U_y,\]
we get a neighbourhood $U$ of $x$ such that $\nu_y(U)<r$ for every $y\in L$, which means that
$U\cap K(r)=\emptyset$. Therefore $K(r)$ is closed.  
\medskip

We now prove (iii) which amounts to checking that for every  $x\in K$ there is $y\in L$ such that $\nu_y(\{x\})\ge m$.

Take any open neighbourhood $U$ of $x$, and a continuous function $g_U:K\to [0,1]$ that
vanishes outside $U$ and $g_U(x)=1$. Then $||g_U||=1$ so  by Lemma \ref{pr:1} there
is ${y(U)}\in L_0$ such that $\nu_{y(U)}(g_U)\ge m$. Let $y$ be a cluster point
of the net $y(U)$, indexed by all open $U\ni x$ (ordered by inverse inclusion). Then  $\nu_y(\{x\})\ge m$. 

Suppose, otherwise, that
$\nu_y(\{x\})< m$; then there is an open set $U\ni x$ such that
$\nu_y(\ol{U})<m$. 
Since $\nu_y$ is a cluster point of $\nu_{y(V)}$,
there must be $V$ such that $x\in V\sub U$ and $\nu_{y(V)}(\ol{U})<m$.
But 
\[ \nu_{y(V)}(\ol{U})\ge \nu_{y(V)}(U)\ge \nu_{y(V)}(V)\ge \nu_{y(V)}(g_V)\ge m,\]
which is a contradiction; the proof of  (iii) is complete.
\end{proof}

\begin{theorem}\label{pe:3}
If $K$ and $L$ are compact spaces such that there is a positive isomorphic embedding
$C(K)\to C(L)$ then $\tau(K)\le \tau(L)$. Moreover, if $L$ is a Frechet (sequential, sequentially compact) space 
then so is the space $K$.
\end{theorem}

\begin{proof}
Note that if $\tau(L)\le \kappa$ and $L_0\sub L$ is closed then $\tau(L_0)\le\kappa$; likewise, the Frechet property
and sequentiality are inhertited by subspaces. Therefore by Lemma \ref{pe:1} we can assue that
there is a positive embedding $T: C(K)\to C(L)$ with $T1_K=1_L$, and the theorem follows from
Lemma \ref{fvm:1} and Proposition \ref{pe:2}(iii).
\end{proof}

For the sake of the next result let us write $ci(L)$ for the class of compact spaces that are continuous images of closed subspaces of 
a given compact space $L$.

\begin{theorem}\label{pe:4}
Let $K$ and $L$ be compact spaces such that there is a positive isomorphic embedding
$C(K)\to C(L)$.   Let $p$ be the least integer such that   
$ 2^p > ||T||\cdot ||T^{-1}||$.

Then there is a sequence 
\[K_1=K\supseteq K_2\supseteq\ldots \supseteq K_p,\] 
of closed subspaces of $K$ such that

\begin{itemize}
\item[(a)] $K_p\in ci(L)$;
\item[(b)] for every $i\le p-1$ and every $x\in K_i\sm K_{i+1}$ there is an open set $U$ containing $x$ such that
$\ol{U}\cap K_i\in ci(L)$.
\end{itemize}
\end{theorem}

\begin{proof}
Take a positive embedding $T:C(K)\to C(L)$ of norm one and write $m=1/||T^{-1}||$. 
As in the previous proof we can assume that $T1_K=1_L$. Following Proposition
\ref{pe:2} we write $\nu_y=T^*\delta_y$ and $\vf_r(y)=\{x\in K: \nu_y(\{x\})\ge r\}$ for $y\in L$ and
$r>0$. 

Let $m_i=2^{i-1}m$ for $i=1,2,\ldots, p$ and put
\[ K_i=\bigcup_{y\in L}\vf_{m_i}(y).\]
Then every $K_i$ is closed by Proposition \ref{pe:2}(ii) and $K_1=K$ by \ref{pe:2}(iii).

To prove (a) note that $m_p>1/2$ by our choice of $p$ so every $\vf_{m_p}(y)$ has at most one element. 
Let $\psi(y)=\vf_{m_p}\cap K_p$ for $y\in L$.
Then $\psi: L\to [K_p]^{\le 1}$ is upper semicontinuous by Lemma \ref{fvm:2} and onto by the definition of $K_p$. 
If $L_0=\{y\in L: \psi(y)\neq\emptyset\}$ then $L_0$ is closed and clearly $\psi$ defines a continuous surjection
from $L_0$ onto $K_p$.

Let $i\le p-1$ and fix $x\in K_i\sm K_{i+1}$. We consider now the mapping  $\psi$ defined as
$\psi(y)=\vf_{m_i}(y)\cap K_i$ for $y\in L$; again $\psi$ is onto and upper semicontinuous.

We claim that $x$ has an neighbourhood $H$ in $K_i$ such that $|\psi(y)\cap H|\le 1$ for $y\in L$. 

Suppose that, otherwise, for every set $H$ which is open in $K_i$ and contains $x$ there is $y_H\in L$ and
distinct $u_H, w_H\in H$ such that $u_H,w_H\in \psi(y_H)$. Then, passing to a subnet we can assume
that $y_H\to y\in L$. Then $\nu_y(\{x\})<m_{i+1}$ since $x\notin K_{i+1}$, so there is open $U$
such that $\nu_y(\ol{U})<m_{i+1}$. Then $\nu_{y_H}(\ol{U})<m_{i+1}$ should hold eventually but
\[\nu_{y_H}(U)\ge \nu_{y_H}(\{u_h,w_h\})\ge 2m_i=m_{i+1},\]  
a contradiction.

Take open $U\sub K$ such that $\ol{U\cap K_i}\sub H$; it follows that $\ol{U}\cap K_i$ is in $ci(L)$, as required. 
\end{proof}

\begin{corollary}\label{pe:5}
If  $T: C(K)\to C(L)$ is a positive embedding then the family of those open $U\sub K$ such that
$\ol{U}\in ci(L)$ is a $\pi$-base of $K$.
\end{corollary}

\begin{proof}
Let $W\sub K$ be a nonempty open set. Taking $K_i$ as in Theorem \ref{pe:4} we infer that there is $i\le p$ such that
the set $W\cap (K_i\sm K_{i+1})$ (setting $K_{p+1}=\emptyset$) has nonempty interior; we conclude applying \ref{pe:4}.  
\end{proof}

\begin{corollary}\label{pe:6}
If  $T: C([0,1]^\kappa)\to C(L)$ is a positive embedding then 
$L$ can be continuously mapped onto $[0,1]^\kappa$. 
\end{corollary}

\begin{proof}
By the previous corollary there is a basic neighbourhood $U$ in $[0,1]^\kappa$ of the form
$U=p_I^{-1}[(a_1,b_1)\times\ldots \times (a_n,b_n)]$, where $p_I:[0,1]^\kappa\to [0,1]^I$
is a projection, such that $\ol{U}\in ci(L)$. Then, taking a continuous surjection
$h:L_0\to\ol{U}$ from a closed subspace $L_0\sub L$, we can extend $h$ to a surjection from $L$ onto
$\ol{U}$, and the latter space is homeomorphic to $[0,1]^\kappa$. 
\end{proof}

Recall that a space $K$ is called homogeneous if for each pair $x_1,x_2\in K$ there is a homeomorphism
$\theta:K\to K$ such that $\theta(x_1)=x_2$.

\begin{corollary}\label{pe:7}
If $T:C(K)\to C(L)$ is a positive embedding and $L$ is Corson compact then $K$ has a $\pi$-base of sets having
Corson compact closures. 

If, moreover,  $K$ is homogeneous then $K$ is Corson compact itself.
\end{corollary}

\begin{proof}
The first part is a consequence of Corollary \ref{pe:5} and the fact that Corson compacta are stable under taking continuous images.

If $K$ is homogeneous then for every $x\in K$ there is an open set $x\in U$ such that $\ol{U}$ is Corson compact.
Hence $K$ can be covered by a finite family of Corson compacta and this easily implies that $K$ is Corson compact too; see e.g.\
Corollary 6.4 in \cite{Pl13}.
\end{proof}

\section{Remarks and questions}

Note that Theorem \ref{pe:4}  implies the following.

\begin{corollary}\label{rq:1}
Suppose that $T: C(K)\to C(L)$ is a positive embedding such that $||T||\cdot ||T^{-1}||<2$.
Then there is a compact subspace $L_1\sub L$ which can be continuously mapped onto $K$.
\end{corollary}

This is, however, a particular case of a result due to Jarosz \cite{Ja84}, who showed that one may drop the assumption of positivity.
It is worth recalling that Jarosz's result was motivated by the following
theorem due to Amir \cite {Am65} and Cambern \cite{Ca67}.

\begin{theorem} \label{in:1}
Suppose that there an isomorphism $T$ from $C(K)$ {\bf onto} $C(L)$ such that $||T||\cdot||T^{-1}||<2$. Then
$K$ is homeomorphic to $L$ and, consequently, $C(K)$ and $C(L)$ are isometric.
\end{theorem} 

Our results from previous sections do not require the condition $||T||\cdot ||T^{-1}||<2$ for the price
of assuming that the operator in question is positive. 
The following elementary example shows that in the setting of Corollary \ref{pe:4} the whole space
$K$ need not be an image of a subspace of $L$.

\begin{example}\label{rq:e}
Let 
\[K=\{x_n: n\ge 1\}\cup\{y_n:n\ge n\}\cup \{x,y\}\]
 consists of two disjoint converging sequences $x_n\to x$, $y_n\to y$ (where $x\neq y$);
let $L=\{z_n: n\ge 0\}\cup\{z\}$, where $z_n\to z$. Define $T:C(K)\to C(L)$ by
\[ Tf(z_0)=f(y),\]
\[ Tf(z_{2n-1})=\frac{f(x_n)+f(y)}{2},\quad  Tf(z_{2n})=\frac{f(x)+f(y_n)}{2}\quad \mbox{ for } n\ge 1.\]
Then $T$ is a positive operator with $||T||=1$; moreover, $T$ is an isomorphism onto $C(K)$.
Indeed, the inverse $S=T^{-1}:C(L)\to C(K)$ is given by
\[Sh(y)=h(z_0), \quad Sh(x)=2h(z)-h(z_0);\]
\[Sh (x_n)= 2h(z_{2n-1}) - h(z_0); \quad Sh(y_n)=2h(z_{2n})-2 h(z)+h(z_0);\]
so $||T^{-1}||\le 5$.

Note that, on the other hand, $K$ is not a continuous image of any subspace of $L$.
\end{example}

It should be stressed that we have not been able to find a complete solution to  Problem \ref{in:2} even for positive embeddings. 

One reason is that Corsonness is not preserved by taking images under finite valued upper semicontinuous functions:
Note for instance that if we let $K$ be the split interval (which is not
Corson compact because every separable Corson compact space is metrizable)
then there is an upper semicontinuous onto mapping $\vf:[0,1]\to [K]^{\le 2}$.

The second reason is that facts like Theorem \ref{pe:4} are not strong enough: Suppose for instance that
$T:C(K)\to C(L)$ is a positive embedding with $||T||\cdot ||T^{-1}||<4$ and $L$ is Corson compact.
Then we conclude form \ref{pe:4} that there is a Corson compact $K_2\sub K$ such that every $x\in K\sm K_2$ has a Corson compact closure.
But a space $K$ with such a property need not be Corson itself --- one can take for instance
$K=[0,\omega_1]$ (the space of ordinals $\alpha\le\omega_1$ with the order topology) and let $K_2=\{\omega_1\}$.

\section{Envelopes of operators between $C(K)$ spaces}

Let $T:C(K)\to C(L)$ be a bounded operator; we can consider the following function 
$e^T:L\to\er$ associated to $T$:
\[e^T(y)=\sup\{Tg(y): g\in C(K), ||g||\le 1\}.\]
We shall call $e^T$ the {\em envelope} of $T$. 

\begin{lemma}\label{res:1}
The envelope $e^T$ of a bounded operator $T:C(K)\to C(L)$ is a lower semicontinuous function
with values in $\left[0,||T||\right]$. If $T$ is an isomorphism onto $C(L)$ then
$1/||T^{-1}||\le e^T(y)\le ||T||$ for every $y\in L$. 
\end{lemma}

\begin{proof}
If $e^T(y_0)>r$ then there is norm-one function $g\in C(K)$, such that
$Tg(y_0)>r$, and then $Tg(z)>r$ for all $z$ from some neighbourhood $V$ of $y_0$ so
$V\sub \{e^T>r\}$. This means that the set of the form $\{e^T>r\}$ is open.

If $T$ is onto $C(L)$ then there is $g\in C(K)$ such that $Tg=1_L$; writing $m=1/||T^{-1}||$ we have
$||mg||\le 1$ so $e^T(y)\ge m$ for every $y\in L$. 
\end{proof}

\begin{lemma}\label{res:2}
For every bounded operator $T:C(K)\to C(L)$ we have $e^T(y)=||T^*\delta_y||$ for every $y\in L$.
\end{lemma}

\begin{proof}
This follows from
\[ ||T^*\delta_y||=\sup_{||g||\le 1} T^*\delta_y(g)= \sup_{||g||\le 1} Tg(y)=e^T(y).\]
\end{proof}

The idea of considering envelopes is based on the observation that the assumption of positiveness of
an embedding $T:C(K)\to C(L)$ might be dropped once we knew that the mapping $y\to ||T^*\delta_y||$
is continuous. In fact the following result explains states that an embedding having a continuous envelope can be, in a sense,
reduced to a positive embedding.

\begin{theorem}\label{res:3}
Suppose that $T:C(K)\to C(L)$ is an isomorphic embedding such that $e^T$ is continuous and positive
everywhere on $L$. Then there exists a positive embedding of $C(K)$ into $C(L\times 2)$.
\end{theorem} 

\begin{proof}
W can assume that $||T||=1$; 
we have $e^T\in C(L)$ and also $h=1/e^T\in C(L)$. If we define $T':C(K)\to C(L)$ by
$T'g=hT(g)$ then $T'$ is also an isomorphic embedding.
If $g\in C(K)$, $||g||=1$ then $|T'g(y)|=|h(y)Tg(y)|\le 1$ so again $||T'||=1$.  
Moreover, it is easy to check that $e^{T'}\equiv 1$.

By the above remark  we can without loss of generality assume that $T:C(K)\to C(L)$ is
such an embedding that $||T||=1$ and $e^T$ is equal to 1 on $L$. 
 As before, for every $y\in L$ write $\nu_y=T^*\delta_y\in M(K).$ By Lemma \ref{res:2} $||\nu_y||=1$
 for every $y\in L$.
 
Every $\nu_y$ can be written as $\nu_y=\nu_y^+-\nu_y^-$, where $\nu_y^+,\nu_y^-\in L^+(K)$.
Since the mapping $L\ni y\to \nu_y$ is continuous, it follows from
the final assertion of Lemma \ref{pr:3} that the mapping $L\ni y\mapsto |\nu_y|$ is also continuous, and, consequently,
so are the mappings $L\ni y\mapsto \nu_y^+$ and $L\ni y\mapsto \nu_y^-$.   

We now define 
\[\theta:L\times 2\to M^+(K),\quad \theta(y,0)=\nu_y^+, \quad  \theta(y,1)=\nu_y^-;\] 
by the above remarks $\theta$ is continuous and  $L'=\theta[L]$ is a compact
subset of $M^+(K)$. 

Let us consider an operator $S:C(K)\to C(L')$ where 
$Sg(\mu)=\mu(g)$ for $\mu\in L'$.
Clearly $S$ is bounded, in fact $||S||\le 1$. 
Let $m=1/||T^{-1}||$; i.e.\ $||Tg|\ge m||g||$ for $g\in C(K)$.
If $g\in B_{C(K)}$ then  $|Tg(y)|\ge m$ for some $y\in L$ so 
$|Sg(\nu_y)|=|\nu_y(g)|\ge m$ which gives that either $|\nu_y^+(g)|\ge m/2$ or
$|\nu_y^-(g)|\ge m/2$. It follows that $||Sg||\ge m/2$ so $S$ is a positive embedding.
Finally, $C(L')$ can be embedded into $C(L\times 2)$ by a positive operator sending
$g\in C(L')$ to $g\circ\theta \in C(L\times 2)$; the proof  is complete.
\end{proof}

We can combine Theorem \ref{res:3} with the result of previous section. 
We denote below by $K+1$ the space $K$ with one isolated point added. Clearly,
$C(K+1)$ is isomorphic to $C(K)\times\er$; recall that except for some pecular spaces
$K$, $C(K+1)$ is isomorphic to $C(K)$, see \cite{Se71}, Koszmider \cite{Ko04}, Plebanek
\cite{Pl04}.

\begin{lemma}\label{res:4}
Let $T:C(K)\to C(L)$ be an isomorphic embedding and suppose that $e^T(y_0)=0$ for some $y_0\in L$.
Then there is  an isomorphic embedding $S:C(K+1)\to C(L)$ with $e^S=e^T+1\ge 1$.
 \end{lemma}

\begin{proof}
Let  $T$ satisfy $m||g||\le ||Tg||\le ||g||$ for $g\in C(K)$.
Put $K+1=K\cup\{z\}$ and  
define $S: K+1\to C(L)$ by 
\[Sf=T(f{\uhr K})+f(z)1_{L}.\] 
Clearly $||T||\le 2$; if $f\in C(K+1)$ is a norm-one function then either $|f(z)|\ge m/2$
which gives $|Sf(y_0)|\ge m/2$, or $|f(z)|<m/2$ which also implies $||Sf||\ge m/2$. Therefore
$S$ is an embedding too. Clearly $e^S=e^T+1\ge 1$.
\end{proof}

\begin{corollary}\label{res:5}
If $K$ and $L$ are compact spaces and there is an isomorphic embedding $T: C(K)\to C(L)$ with
a continuous envelope then $\tau(K)\le \tau(L)$. Moreover, the Frechet property (sequentiality, sequential compactness) of $L$ implies that $K$ is Frechet (sequential, sequentially compact).
\end{corollary}

\begin{proof}
If $e^T$ is positive on $L$ then by Theorem \ref{res:3} and Proposition \ref{pe:2} $K$ is an image
of $L\times 2$ under a finite-valued upper semicontinuous set-function; we conclude applying 
Lemma \ref{fvm:1}.

If $e^T$ is zero at some point then we first use Lemma \ref{res:4} and continue with $K+1$ replacing
$K$.
\end{proof}

\end{document}